\documentclass{amsart}
\usepackage{amsmath, amssymb,epic,graphicx,mathrsfs,enumerate}
\usepackage[all]{xy}

\usepackage{amsthm}
\usepackage{amssymb}
\usepackage{latexsym}
\usepackage{longtable}
\usepackage{epsfig}
\usepackage{amsmath}
\usepackage{hhline}

 \DeclareMathOperator{\perm}{Sym}
 
\DeclareMathOperator{\aut}{Aut} 
 \DeclareMathOperator{\M}{M}
\DeclareMathOperator{\po}{P\Omega}

 \DeclareMathOperator{\frat}{Frat}
\DeclareMathOperator{\ssl}{SL} \DeclareMathOperator{\psl}{PSL}
\DeclareMathOperator{\psp}{PSp} 

\DeclareMathOperator{\sym}{Sym}

\DeclareMathOperator{\GL}{GL}

\DeclareMathOperator{\alt}{Alt}
\DeclareMathOperator{\End}{End} \DeclareMathOperator{\h}{{H^1}}
 \DeclareMathOperator{\diag}{diag}

\renewcommand{\emptyset}{\varnothing}

\newtheorem{thm}{Theorem}
\newtheorem{cor}[thm]{Corollary}
 \newtheorem{lemma}[thm]{Lemma}
\newtheorem{prop}[thm]{Proposition} 
 
\newtheorem{question}[]{Question} 
\numberwithin{equation}{section}

\renewcommand{\footnote}{\endnote}
\newcommand{\ignore}[1]{}\makeglossary
\begin{document}
\bibliographystyle{amsplain}
\subjclass{20F05}
\title{Invariable generation  with elements\\ of coprime prime-power order}
\author{Eloisa Detomi and Andrea Lucchini}
\address{
Eloisa Detomi and Andrea Lucchini, Universit\`a degli Studi di Padova,  Dipartimento di Matematica, Via Trieste 63, 35121 Padova, Italy}


\begin{abstract}A finite group $G$ is \emph{coprimely-invariably generated} if there exists a set of generators $\{g_1, \ldots , g_d\}$ of $G$ with the property that the orders $|g_1|, \ldots , |g_d|$ are pairwise coprime and that
 for all $x_1, \ldots, x_d \in G$ the set $\{g_1^{x_1}, \ldots, g_d^{x_d}\}$ generates $G$. In the particular case when $|g_1|, \ldots , |g_d|$ can be chosen to be prime-powers we say that $G$ is
\emph{prime-power coprimely-invariably generated}. We will discuss these properties,
 proving also that the second one is stronger than the first,
but that in the particular case of finite soluble groups they are equivalent.
\end{abstract}

\maketitle

\section{Introduction}

Following  \cite{dixon} and \cite{ig}, we say that
a subset $\{g_1, \ldots , g_d\}$ of a finite group $G$  invariably generates $G$ if
 $\{g_1^{x_1}, \ldots , g_d^{x_d}\}$ generates $G$ for every choice of $x_i \in G$.
Then we say that a finite group $G$ is \emph{coprimely-invariably generated} (CIG, for short) if there exists a
 subset $X=\{g_1, \ldots , g_d\}$ of $G$ that invariably generates $G$ and with the property that the orders $|g_1|, \ldots , |g_d|$ are pairwise coprime; in the particular case when  $|g_1|, \ldots , |g_d|$ can be chosen to be prime-powers we say that $G$ is
\emph{``prime-power'' coprimely-invariably generated} (PCIG, for short). Using the classification of the finite
simple groups, we will prove in Theorem \ref{simple} that all the finite non abelian simple groups are PCIG.

\

Clearly a PCIG group is in particular coprimely-invariably generated.
In Proposition \ref{s_7} we will prove the existence of a finite CIG group $G$ which is not PCIG.
However we obtain the unexpected result that a finite soluble CIG group is PCIG.
Actually we prove a stronger result:
\begin{thm}\label{CIG=PCIG}Let $G$ be a finite soluble
group and assume that there exists a set $\{g_1,\dots,g_d\}$  of invariable generators of $G$ with the property that
 $|g_1|, \ldots , |g_d|$ are pairwise coprime.
For each $i\in \{1,\dots, d\},$ we write $g_i=x_{i,1}\cdots x_{i,u_i}$ where the $x_{i,j}$ are  powers of $g_i$ of coprime prime-power order.
Then the set $$X=\{x_{i,j} \mid 1\leq i \leq d, 1\leq j \leq u_i\}$$ invariably generates $G.$
\end{thm}

A key role in the proof of Theorem \ref{CIG=PCIG} is played  by the theory of crowns, introduced by
Gasch\"{u}tz in \cite{Ga}. Using the properties of the crowns,
it can be proved (see Proposition \ref{reducecrown}) that a subset $\{g_1,\dots,g_d\}$
of a finite soluble group $G$ invariably generates $G$  if and only if $\{\phi(g_1),\dots,\phi(g_d)\}$
 invariably generates $\phi(G)$ for any epimomorphism $G\to V^t \rtimes H,$ where $H$ acts in the
same way on each of the $t$ direct factors and the action of $H$ on $V$ is faithful and irreducible.
Thus, a crucial step is to study, for a semidirect product  $V^t \rtimes H,$ under
which conditions a coprime-invariable generating set of $H$ can be lifted to a  coprime-invariable generating set of $V^t \rtimes H.$
Results is this direction are given in Proposition \ref{matrici} and Proposition \ref{ltcase}.
Actually,  Proposition \ref{matrici} deals with the more general case of invariable generation, and might have
some relevance in the study of invariable generation of soluble groups.
For example, using this result, we can improve the bound given in \cite[Theorem 3.1]{ig}
 for the smallest cardinality $d_I(G)$ of an invariably generating set of a finite
group $G$: when $G$ is soluble, the authors of \cite{ig} prove that $d_I(G)\leq a,$
where $a$ is the length of a chief series of $G.$ Our stronger bound is:
\begin{thm}\label{sol}If $G$ is a finite soluble group
then $$d_I(G)\leq \sum_{A\in \mathcal V}\left \lceil  \frac{\delta_G(A)}{\dim_{\End_G(A)}A} \right \rceil$$
where $A$ runs in the set $\mathcal V$ of representatives for the irreducible $G$-groups that are
$G$-isomorphic to a complemented chief factor of $G$ and, for $A\in \mathcal V,$ $\delta_G(A)$
denote the number of complemented factors $G$-isomorphic to $A$ in a chief series of $G.$
\end{thm}

Using Proposition \ref{ltcase},
 as a byproduct we obtain a characterization of the finite supersoluble CIG groups.

\begin{thm}\label{super}Let $G$ be a finite supersoluble group.  Then $G$ is coprimely-invariably generated
if and only if no chief series of $G$ contains two complemented and $G$-isomorphic chief factors.
\end{thm}

A motivation for our interest in  PCIG groups
is due to the fact a group
 without proper subgroups with the same exponent (we will call
these groups {\sl{minimal-exponent groups}}) is actually a PCIG group. Indeed assume that
$G$ is a minimal-exponent group with $e\!\!:=\exp(G)=p_1^{n_1}\cdots p_t^{n_t}$. Then for every $i,$ the group
$G$ contains an element $g_i$ of order $p_i^{n_i}$. Clearly $\exp \langle g_1^{x_1},\dots,g_t^{x_t} \rangle=e$,
for every $x_1,\dots,x_t\in G$, so $G$ is a PCIG group. On the other hand, there are PCIG groups
 which are not minimal-exponent.
 An  easy  example  is given by the semidirect product  $G$ of two copies $C_3 \times C_3$ of  a cyclic group of order $3$  and a cyclic group $\langle y \rangle$ of order $2$, acting trivially on the first component of $C_3 \times C_3$ and non-trivially on the second one. This group $G$ can be invariably  generated by $y$ and a non-trivial diagonal element $(x,x)$ in $C_3\times C_3.$ So $G$ is PCIG, but has a proper (cyclic) subgroup, $\langle (x,1)y \rangle$, with the same exponent.

\

In \cite{CIG} the authors prove that if $G$ is a CIG group, then the minimal size $d(G)$ of a generating set is at most $ 3$, with $d(G) \leq 2$ whenever $G$ is soluble.
In particular, the same results hold for PCIG and minimal-exponent groups. As it is noticed in
\cite{CIG} the bound $d(G)\le 3$ is sharp for both PCIG and CIG group. But it remains an open  and difficult question
to decide whether there exists a minimal-exponent group which is not 2-generated. We think that, in order to collect more information on this problem, it could be useful to investigate in more generality how far is the property of being minimal-exponent from the PCIG property.
A substantial difference between the two properties is that only the second one is inhered by the epimorphic images, while
an epimorphic image of a minimal-exponent group satisfies PCIG but is not in general minimal-exponent.
Really we do not know examples of a PCIG group that does not appear as an epimorphic image of a
minimal-exponent group and we propose the following question: is it true that for every  finite PCIG group $X$ there exists a minimal-exponent finite group $Y,$ and a normal subgroup $N$ of $Y,$ such that $X$ is isomorphic to $Y/N?$

  The structure of the paper is the following. In Section \ref{Soluble} we investigate under which conditions a set of invariable generators of an epimorphic image can be lifted to the whole group; then we construct a CIG-group that is not PCIG and prove Theorem \ref{CIG=PCIG}.
  Theorem \ref{sol} is proved in Section \ref{invariable}.
  In Section \ref{Supersoluble} we prove Theorem \ref{super} and we analyse the structure of supersoluble minimal-exponent and CIG-groups.
 Finally, in  Section \ref{quest} we prove that all the finite simple groups are PCIG and
 we present some further
questions and  examples.


\section{CIG and PCIG groups}\label{Soluble}

As noted above, a PCIG group is a CIG group.
In this section we prove that there exists a CIG group which is not PCIG, but the two properties coincide on the class of soluble groups.

We start with some definitions and notations. In the following all groups are finite.
A subset $X=\{g_1, \ldots , g_d\}$ of $G$ is an invariable generating set if
 $\{g_1^{x_1}, \ldots , g_d^{x_d}\}$ generates $G$ for every choice of $x_i \in G$.
For brevity, we say that  the subset $X=\{g_1, \ldots , g_d\}$ is a
CIG-set of $G$
 if $X$ invariably generates $G$ and the elements $g_1, \ldots ,g_d$ have coprime orders;
similarly, we say that a  subset $X=\{g_1, \ldots , g_d\}$ of $G$ is a PCIG-set of $G$
if $X$ is a CIG-set of $G$ and the orders $|g_1|, \ldots , |g_d|$ are prime-powers.

\

\noindent{\bf Notation:}
 Let $H$ be a  group acting
irreducibly and faithfully on an elementary abelian  $p$-group $V$.
For a positive integer $u$ we consider
 the semidirect product $G = V^u \rtimes H$: unless otherwise stated, we assume that  the action of $H$ is diagonal on $V^u$, that is, $H$ acts in the
same way on each of the $u$ direct factors.

\begin{lemma}\label{inizio}
The group $G=V^u \rtimes H$
 is comprimely-invariably generated if and only if there exist a CIG-set $\{h_1,\dots,h_d\}\in H$ and an element $v \in V^u$ such that
\begin{enumerate}
\item  $(|h_i|,p)$ if $i\neq 1.$
\item $\{h_1v,h_2,\dots,h_d\}$ is a CIG-set for $G$.
\end{enumerate}
\end{lemma}

\begin{proof}Assume that $\{g_1,\dots,g_d\}$ is a CIG-set for $G=V^u\rtimes H$: thus $\{g_1,\dots,g_d\}$
 invariably generates $G$ and  $|g_1|, \ldots , |g_d|$ are coprime, in particular
  we can assume that $|g_i|$ is prime to $p$ for every $i \neq 1$.
For every $i$, we write $g_i=x_iv_i$ where $v_i \in V$ and $x_i \in H$. If $i \neq 1$,
  then $g_i$ is conjugate to an element $h_i$ of $\langle x_i \rangle$, since  $\langle x_i \rangle$ is a Hall $p'$-subgroup of $V \langle x_i \rangle$
 and $g_i$ is a $p^\prime$-element.
As  $\{g_1,\dots,g_d\}$   invariably generates $G,$ the set  $\{x_1v_1,h_2,\dots,h_d\}$ is a CIG-set for $G$.
Then, it is not difficult to see that $\{x_1, h_2, \dots,h_d\}$
 must be a CIG-set for $H.$
 \end{proof}

\begin{thm}\label{abcase}
Let $G=V^u \rtimes H$, where $H$ acts irreducibly and faithfully on an elementary abelian  $p$-group $V$ and let $F=\End_H(V).$
If $G$ is coprimely-invariably generated, then $$u\leq \max_{h\in \Lambda}\dim_FC_V(h)$$
where $\Lambda$ is the set of
elements $h$ in $H$ contained in some CIG-set $\{h,h_2,\dots,h_d\}$ of $H$  with  $|h_i|$ prime to $p$ for every $i\in \{2,\dots,d\}$.
\end{thm}

\begin{proof}
If $G$  is coprimely-invariably generated, then there exist $h_1,\dots,h_d$ and $v$ as in the statement of Lemma \ref{inizio}.
 Let $h=h_1$ and choose $w \in V$. Then $(hv)^w=h^wv=h[h,w]v$, where $ [h,w]v \in V$.
Since
 \[ G= \langle  (hv)^w, h_2,\dots,h_d \rangle \leq \langle  [h,w]v, h, h_2,\dots,h_d \rangle
 \leq \langle [h,w]v \rangle ^H H \]
 we get that $ \langle [h,w]v \rangle ^H=V^u$, that is,
  $[h,w]v$ is a cyclic generator for the $H$-module $V^u$.
  Let $v=(v_1, \ldots ,v_u)$ and $w=(w_1, \ldots  , w_u)$. Switching to additive notation,  the fact that $v+[h,w]$ is a  generator for the $H$-module $V^u$ implies
  that
 the elements $v_1+[h,w_1], v_2+[h,w_2], \ldots , v_u+[h,w_u]$ are linearly independent in the  $F$-vector space $V$.
 Let $\alpha_1,\dots,\alpha_u \in F.$
 Note that
 \[ \sum_{i=1}^{u} \alpha_i (v_i+[h,w_i])=0 \]
 if and only if
$$ \sum_{i=1}^{u} \alpha_i v_i \in \left[h,\sum_{i=1}^{u} \alpha_i w_i\right].$$
Since this condition  holds for every choice of  $w \in V^u$,
  we deduce that
  $ v_1, \ldots ,v_u$ have to be linearly independent modulo the subspace $[h,V]$, hence
  $$u \leq \dim_{F}V -\dim_{F}[h,V]=\dim_{F}C_V(h).$$
 In particular $u\leq \max_{h\in \Lambda}\dim_{F}C_V(h)$, as required.
\end{proof}

As an application of the previous result, we now prove that the properties CIG and PCIG are not equivalent.

\begin{prop}\label{s_7}
There exists a coprimely-invariably generated group which is not PCIG.
\end{prop}
\begin{proof}
Note that $H=\sym(7)$ admits an absolute irreducible module $V$ of order
$7^5$ (the full deleted permutation module).
We use Theorem \ref{abcase} to deduce that
$G=V^5\rtimes H$ is not PCIG. Indeed a set of elements of $H$ of prime-power order that invariably generates $H$ must contains
an element $x$ of order 7 (since all the other elements of $H$ are contained in a point-stabilizer)
and $\dim_{F_7}C_V(x)\leq 4$, since the action is faithful.
Thus it follows from  Theorem \ref{abcase} that   $G$  is not PCIG.

However, if $a_1,\dots,a_5$ are $F_7$-linearly independent elements of
$V,$ then $z=(a_1,\dots,a_5)$ generates $V^5$ as an $H$-module and $\{z,(1,2,3,4,5), (1,2,3)(4,5,6,7)\}$
is a CIG-set of $G.$ Thus  $G$ is a (non-soluble)  CIG group.
Note that in this case the trivial element $e$ is contained in the CIG-set $\{e,  (1,2,3,4,5), (1,2,3)(4,5,6,7)\}$ of $H$, and $ \dim_{F_7}C_V(e)=5$, so that the condition of  Theorem \ref{abcase} is satisfied.
\end{proof}

Theorem \ref{abcase} gives a necessary condition on $u$ in order to ensure that
the semidirect product $G=V^u\rtimes H$ is a CIG group given that the  irreducible
linear group $H$ is CIG. We are going to prove that this is indeed also
a sufficient condition whenever $\h(H,V)=0.$

\

For the convenience of the reader, we provide the proof of the following key result:

\begin{prop}\label{crit}\cite[Proposition 2.1]{dw}. Suppose $\h(H,V)=0$ and  $H=\langle h_1, \dots, h_d\rangle$.
 Let $w_i=(w_{i,1},\dots,w_{i,u})\in V^u$ with $1\leq i\leq d.$ The following are equivalent.
\begin{enumerate}
\item $G \neq \langle h_1w_1,\dots,h_dw_d\rangle$;
\item there exist $\lambda_1,\dots,\lambda_u \in F=\mathrm{End}_{H}(V)$ and $w \in V$
with $(\lambda_1,\dots,\lambda_u, w) \not= (0, \ldots,0,0)$ such
that $\sum_{1\leq j\leq u} \lambda_j w_{i,j} = [h_i,w]$ for each $i\in\{1,\dots,d\}.$
\end{enumerate}
\end{prop}
\begin{proof}
Let  $K=\langle h_1w_1,\dots,h_dw_d\rangle$. First we  prove, by induction on $u$, that if $K\neq G$
then $(2)$ holds. Let $z_i=h_i(w_{i,1},\dots,w_{i,u-1},0)$ and let  $Z=\langle z_1, \dots, z_d
\rangle$. If $Z \not\cong V^{u-1}H$, then, by induction,
there exist $\lambda_1,\dots,\lambda_{u-1} \in F$ and $w \in V$
with $(\lambda_{1}, \ldots , \lambda_{u-1},w) \not=
(0,\ldots,0,0)$ such that $\sum_{1\leq j\leq u-1} \lambda_j w_{i,j} = [h_i,w]$ for each $i\in\{1,\dots,d\}.$ In this case
$\lambda_1,\dots,\lambda_{u-1},0$ and $w$ are the requested
elements.

 So we may assume $Z \cong V^{u-1}H$. Set
$V_u=\{(0,\dots,0,v)\mid v \in V\}$. We have $ZV_u = KV_u=G$
and $Z \neq G$; this implies that $Z$ is a complement of $V_u$ in
$G$ and therefore there exists $\delta \in {\rm{Der}}(Z,V_u)$
such that $\delta(z_i)=w_{i,u}$ for each $i\in\{1,\dots,d\}.$ However,
by Propositions 2.7 and 2.10 of \cite{AG}, we have
${\rm{H}}^{1}(Z,V_u)\cong F^{u-1}$. More precisely if $\delta
\in {\rm{Der}}(Z,V_u)$, then there exist an inner derivation
$\delta_w \in {\rm{Der}}(H,V)$ (for $w\in V$)
 and $\lambda_1,\dots,\lambda_{u-1}
\in F$ such that for each $h(v_1,\dots,v_{u-1},0) \in Z$ we
have
$$ \delta(h(v_1,\dots,v_{u-1},0))=\delta_w(h)+\lambda_1v_1 +
\dots+ \lambda_{u-1}v_{u-1}=[h,w]+\lambda_1v_1 +
\dots+\lambda_{u-1}v_{u-1}.$$
 In particular $-\sum_{1\leq j\leq u-1} \lambda_j w_{i,j}+w_{i,u}=[h_i,w]$ for each $i\in\{1,\dots,d\},$
 hence  $(2)$ holds.

Conversely, if (2) holds then $\langle h(v_1,\dots,v_u)\mid
[h,w]=\lambda_1v_1+\dots +\lambda_uv_u\rangle$ is a proper subgroup
of $G$ containing $K$.
\end{proof}

Let $n$ be the dimension of $V$ over $F=\mathrm{End}_{H}(V)$. We have an injective homomorphism from $H$ to
$GL(n,F)$: we will use the notation $\overline h$ to denote the image of $h\in H$ under this homomorphism.
We will use the additive notation for $V=F^n$ and we will identify its elements
with $1 \times n$ matrices  with coefficient in $F.$ With this notation, if
$v\in V$ and $h\in H$, then $v^h$ is the $1 \times n$ matrix  ${v}\overline{h}$
obtained using the matrix multiplication; in particular, $[h,v]=v-v\overline{h}=v(1-\overline{h})$.

Let $\pi_i:V^u \mapsto V$ the canonical projection on the $i$-th component:
 $\pi_i(v_1, \dots , v_u)=v_i$.

If  $w_i=(w_{i,1},\dots,w_{i,u})\in V^u$, $i=1, \dots ,d$, consider the vectors
$$r_j= (\pi_j(w_1), \dots , \pi_j(w_d))=(w_{1,j}, \dots, w_{d,j})  \in V^d, j=1, \dots ,u.$$
Then  Proposition \ref{crit} says that  the elements $h_1w_1,\dots,
h_dw_d$ generate a proper subgroup of $G$ if and only
if
there exists a non-zero vector $(\lambda_{1}, \ldots ,
\lambda_{u}, w)$ in $F^{u} \times V$ such that
$$
\sum_{1\leq j\leq u} \lambda_j r_j = \big([h_1,w], \dots , [h_d,w]\big).
$$

This is equivalent to saying that there exist elements $w_1,\dots,w_d$
in $V^u$ such that $\langle h_1w_1,\dots,h_dw_d \rangle=
G$ if and only if
 there exist elements
$r_1, \dots , r_u$ in $V^d$ that are linearly independent modulo the vector space
$$D=\{ \big([h_1,w], \dots , [h_d,w]\big)\in V^d \mid w\in V\}.$$
Notice that $\dim_F(D)=\theta\dim_F(V)$ where $\theta=0$ or $1$ according to whether $V$ is a trivial $H$-module or not:
 indeed, if $H$ is non-trivial,
then the linear map $\alpha: V \to V^d$, $w \mapsto \big([h_1,w], \dots , [h_d,w]\big)$ is injective (if $w \in
\ker \alpha$ then $[h_i, w]=w$ for each $i\in\{1,\dots,d\}$, against the fact that ${h_1},\dots,{h_d}$
generate a non-trivial irreducible group)
 Therefore,  there exist elements $w_1,\dots,w_d$
in $V^u$ such that $\langle h_1w_1,\dots,h_dw_d \rangle=
G$ if and only if
$u \leq \dim_F (V^d)-\dim_F(D)=nd-\theta n$.

We now discuss the same question in the case of invariable generation: we are
going to prove that $G=V^u\rtimes H$ admits a set  of
cardinality $d$ invariably generating $G$  if and only if, in the previous notations, there exists
a set $\{h_1,\dots,h_d\}$ invariably generating $H$ with the property
that $u\leq nd-\sum_{1\leq i \leq d}\dim_F{[h_i,V]}=\sum_i \dim_{\End_H(V)}C_V(h_i).$

\begin{prop}\label{matrici}
Suppose that $h_1,\dots,h_d$ invariably generate $H$ and that $\h(H,V)=0.$
 Let $w_1,\dots,w_d\in V^u$ with
$w_i=(w_{i,1},\dots,w_{i,u})$. For $j\in \{1,\dots,u\}$, consider the vectors
$$r_j=\big(\pi_j(w_1), \dots , \pi_j(w_d)\big)=(w_{1,j}, \dots, w_{d,j})\in V^{d}.$$
Then $h_1w_1, h_2w_2,\dots,h_d w_d$ invariably generate
 $V^u\rtimes H$ if and only if the vectors $r_1,\dots,r_u$ are linearly independent modulo
 $$W=\{(u_1,\dots,u_d)\in V^{d}\mid  u_i \in [h_i,V], \ i=1, \dots , d\}.$$
In particular, there exist some elements $w_1,\dots,w_d\in V^u$ such that $h_1w_1, h_2w_2,\dots,h_d w_d$ invariably generate
  $V^u\rtimes H$ if and only if
$$u\leq nd-\dim W=\sum_i \dim_{\End_H(V)}C_V(h_i).$$
\end{prop}

\begin{proof}
Let $g_i=y_ix_i$
 with $x_i\in H$ and $y_i=(y_{i,1},\dots,y_{i,u})\in V^u$ for $i\in\{1,\dots,d\}$
and let $X_{g_1,\dots, g_d}=\langle (h_1w_1)^{g_1},\dots, (h_dw_d)^{g_d}\rangle.$
We have
$$(h_iw_i)^{g_i}=(h_i^{y_i}w_i)^{x_i}=
h_i^{x_i}([h_i,y_i]+w_i)\overline{x_i}=h_i^{x_i}z_i$$
where
 $z_i=([h_i,y_i]+w_i)\overline{x_i} \in V^u$.
 Then $X_{g_1,\dots, g_d}=G$ if and only if the vectors of $V^d$
 $$\big(\pi_j(z_1), \dots , \pi_j(z_d)\big)=
\big(   ([h_1,y_{1,j}]+w_{1,j})\overline{x_1}, \dots , ([h_d,y_{d,j}]+w_{d,j})\overline{x_d}\big),
 $$
 $j=1,\dots,u$,
are linearly independent modulo the subspace
\begin{eqnarray*}
\tilde{D}&=&\{ \big([h_1^{x_1},w], \dots , [h_d^{x_d},w]\big)\in V^d \mid w\in V\}\\
&=& \left\{\left(\left[h_1,w\overline{x_1^{-1}}\right] \overline{x_1}, \dots , \left[h_d,w\overline{x_d^{-1}}\right] \overline{x_d} \right)\in V^d \mid w\in V\right\}.
\end{eqnarray*}
Note that the map $f_{(x_1, \dots, x_d)}:V^d \mapsto V^d$ defined by
 $$f_{(x_1, \dots, x_d)}(v_1, \dots , v_d)=(v_1 \overline{x_1}, \dots , v_d \overline{x_d})$$ is an isomorphism.
Therefore the previous condition is equivalent to have that the vectors
 $$\big(   [h_1,y_{1,j}]+w_{1,j}, \dots , [h_d,y_{d,j}]+w_{d,j}\big)=r_j+\big( [h_1,y_{1,j}],  \dots , [h_d,y_{d,j}]\big),
 $$
 for $j=1,\dots,u,$
 are linearly independent modulo
  the subspace
$${D}=\left\{\left(\left[h_1,w\overline{x_1^{-1}}\right], \dots ,\left[h_d,w\overline{x_d^{-1}}\right]\right)\in V^d \mid w\in V \right\}.$$
Since this condition has to hold for every choice of  $y_i \in V^u$ and $x_j \in H$,  this means that
 the elements
 ${r_1}, \dots , {r_u}$ have to be
linearly independent modulo
  the subspace
$W=\{(u_1,\dots,u_d)\in V^{d}\mid  u_i \in [h_i,V], \ i=1, \dots , d\},$ as required.
\end{proof}

\noindent{\bf{Example.}} Assume $H=\GL(2,2)\cong\perm(3)$ and let $V=F_2\times F_2.$
Suppose that $(h_1,\dots,h_d)$ is a sequence of invariable generators of $H$ where, let say,
there are  $a$ entries of order 2, $b$ entries of
order 3 and $c$ entries equal to the identity element, with $a+b+c=d.$
 Since $h_1,\dots,h_d$  invariably generate $H$, then necessarily $a\ge 1$ and $b \ge 1$.
 According to the previous proposition, this sequence
can be lifted to a sequence of invariable generators of  $V^u\rtimes H$
if and only if $u\leq a+2c.$ Since $a+2c=a+2(d-a-b)=2d-a-2b$ has its maximal value at
$a=b=1$ and $c=d-2$, 
 $V^u\rtimes H$ can be invariably generated by
$d$ elements if and only if $u\leq 1+2(d-2)=2d-3.$
 In particular $G:=V^2\rtimes H$ cannot be invariably
generated by 2 elements and consequently, since the prime divisors of $|G|$
are only 2 and 3, $G$ has no coprime-invariable generating set, hence it is not CIG.

\

The previous proposition explains when and how a set of invariable generators
for $H$ can be lifted to a set of invariable generators for $V^u\rtimes H.$
The situation changes if we consider the same question for coprime-invariable generating sets:
in this case, by Lemma \ref{inizio}, when we try to lift a CIG-set $\{h_1,\dots,h_d\}$ of $H$
to a CIG-set of $V^u\rtimes H$ we are free to modify via multiplication by elements
of $V^u$ only one of the $d$ generators and consequently we get a more restrictive condition,
as described by the following proposition.

\begin{prop}\label{ltcase}Suppose that $h_1,\dots,h_d$ have coprime order and invariably generate  $H$ and that $\h(H,V)=0.$
 Assume $h_i$ is a $p'$-element for $i\neq 1$. Let $v=(v_{1},\dots,v_{u})\in V$. Then
$h_1v,h_2,\dots,h_d$  invariably generate  $G=V^u\rtimes H$ if and only if
the vectors $v_{1},\dots,v_{u}$ are linearly independent modulo $[h_1,V].$
\end{prop}

\begin{proof}If follows from the proof of Theorem \ref{abcase} that the condition
that the vectors $v_{1},\dots,v_{u}$ are linearly independent modulo $[h_1,V]$
is necessary to have that $h_1v,h_2,\dots,h_d$  invariably generate  $G.$
On the other hand, in the notation of Proposition \ref{matrici},  if
the vectors $v_{1},\dots,v_{u}$ are linearly independent modulo $[h_1,V],$
then $v_1^*=(v_1,0,\dots,0),\dots,v_u^*=(v_u,0,\dots,0)$ are linearly independent
modulo $W$, hence $h_1v,h_2,\dots,h_d$ is a CIG-set of $G.$
\end{proof}

Now we apply the theory of crowns to reduce the problem of invariable generation for a soluble groups $G$
to the particular case of a semidirect product $V^u\rtimes H$, as studied above.

Let $G$ be a finite soluble group, and let $\mathcal V$ be a set
of representatives for the irreducible $G$-groups that are
$G$-isomorphic to a complemented chief factor of $G$. For $V \in
\mathcal V$ let $R_G(V)$ be the smallest normal subgroup contained
in $C_G(V)$ with the property that $C_G(V)/R_G(V)$ is
$G$-isomorphic to a direct product of copies of $V$ and it has a
complement $H$ in $G/R_G(V)$.
The factor group  $C_G(V)/R_G(V)$ is usually called ``crown'' of $G$ associated to the $G$-module $V$
 (for more details, see \cite{classes}).
The non-negative integer
$\delta_G(V)$ defined by $C_G(V)/R_G(V)\cong_G V^{\delta_G(V)}$ is
called the $V$-rank of $G$ and it coincides with the number of
complemented factors in any chief series of $G$ that are
$G$-isomorphic to $V.$  Actually, $G/R_G(V) \cong V^{\delta_G(V)}  \rtimes H$, where $H$ acts diagonally on each component of $V^{\delta_G(V)}$
(see \cite{Ga}).

\begin{prop}\label{reducecrown}
\label{lemma}\cite[Proposition 2.4]{am} Let $G$ and $\mathcal V$ be as above. Let $x_{1},
\ldots , x_{d}$ be elements of $G$ such that $\langle
x_1,\dots,x_d,R_G(V)\rangle=G$ for any $V \in \mathcal V$. Then
$\langle x_1,\dots,x_d\rangle=G$.
\end{prop}


The previous result allows us to reduce the proof of Theorem \ref{CIG=PCIG}  to the quotient groups $G/R_G(V)$.

\begin{proof}[Proof of Theorem \ref{CIG=PCIG}] Let $G$ be a soluble group.
We want to prove that if  $\{g_1,\dots,g_d\}$ is a CIG-set of $G$,
 then the set
 $X=\{x_{ij} \mid 1\leq i \leq d, 1\leq j \leq u_i\}$ is a PCIG-set of $G$, where the $ x_{ij}$, $1\leq j \leq u_i$,
 are the  powers of $g_i$ of prime power order.

The proof is by induction on $|G|,$ case $|G|=1$ being trivial.
By Proposition \ref{reducecrown} it is not restrictive to assume that $G=V^u\rtimes H$ where $H$ is an irreducible subgroup of $\GL(V)$ and  $V$ is a $p$-group.
 Note that, as $G$ is soluble,  $H^1(H,V)=0$ \cite[Lemma 1]{st}.
 We may assume that $p$ does not divide $|g_i|$ if $i\neq 1$: in particular it is not restrictive to
assume $g_i \in H$ if $i\neq 1.$
Moreover it is not restrictive to assume that $g_1=abv$ with $a,b \in H,$  $v \in V$,
$a$ and $bv$ elements of $\langle g_1 \rangle$, $(|a|,p)=1,$ and $|bv|$ a $p$-power.

Let $v=(v_1,\dots,v_u).$
 As $ab, g_2,\dots,g_d$ invariably generate $H$, by Lemma \ref{ltcase},
the vectors $v_1,\dots,v_u$ are linearly independent modulo $[ab,V].$ We may assume $x_{1,1}=bv$ and
$x_{i,j}\in H$, $(|x_{i,j}|,p)$
if $(i,j)\neq (1,1).$

By induction
$$\{b,x_{i,j}\mid 1\leq i \leq d, 1\leq j \leq u_i, (i,j)\neq (1,1)\}$$
 is PCIG-set of $H.$ Since $bv \in \langle abv \rangle,$ there exists $n\in \mathbb N$
such that $b=(ab)^n.$ We may identify $ab$ and $b$ with two matrices $X$ and $Y$ such that $Y=X^n.$
 In this identification $[ab,V]$
correspond to the image $V(1-X)$ of the linear map $1-X$ and $[b,V]$ to the image $V(1-Y)$ of the linear map $1-Y.$ Clearly
we have $$V(1-Y)=V(1-X^n)=V(1+X+\dots+X^{n-1})(1-X)\leq V(1-X),$$
i.e. $[b,V]\leq [ab,V]$, hence
the vectors $v_1,\dots,v_u$, being linearly independent modulo $[ab,V],$
are linearly independent modulo $[b,V]$ too. By Lemma \ref{ltcase} we conclude that  the set
$$\{bv,x_{i,j}\mid 1\leq i \leq d, 1\leq j \leq u_i, (i,j)\neq (1,1)\}=\{x_{i,j} \mid 1\leq i \leq d, 1\leq j \leq u_i\}$$
 invariably generates $G,$ and is actually  a PCIG-set of $G$, as required.
\end{proof}


\section{A bound for $d_I(G)$ when $G$ is a soluble group}\label{invariable}
In this section we will prove Theorem \ref{sol}
that bounds the smallest cardinality $d_I(G)$ of an invariably generating set of a finite
 soluble  group $G$ by a function of the  numbers  $\delta_G(A)$ of complemented factors $G$-isomorphic to $A$ in a chief series of $G$,
where $A$ runs in the set $\mathcal V$ of representatives for the irreducible $G$-groups that are
$G$-isomorphic to a complemented chief factor of $G$.

First we need to recall the following results:

\begin{lemma}{\cite[Lemma 1.3.6]{classes}}\label{corona}
Let $G$ be a finite soluble group with trivial Frattini subgroup. There exists
a crown  $C/R$ and a non trivial normal subgroup $U$ of $G$ such that $C=R\times U.$
\end{lemma}

\begin{lemma}{\cite[Proposition 11]{crowns}}\label{sotto} Assume that $G$ is a finite soluble group with trivial Frattini subgroup and let $C, R, U$ be as in the statement of Lemma \ref{corona}. If $KU=KR=G,$ then $K=G.$
\end{lemma}

\begin{proof}[Proof of Theorem \ref{sol}] The proof is by induction on $|G|.$
Let $$\eta=\eta(G)=\sum_{A\in \mathcal V}\left \lceil  \frac{\delta_G(A)}{\dim_{\End_G(A)}A} \right \rceil.$$
If $\frat(G) \neq 1$, then by induction $d_I(G/\frat G)\leq\eta(G/\frat(G))$. As $$\eta(G/\frat(G))=\eta(G) \text{ and } d_I(G/\frat G))=d_I(G),$$
we immediately conclude $d_I(G)\leq \eta.$

 Assume now $\frat(G)=1.$ In this case, by Lemma \ref{corona}, there exists a crown $C/R$ of $G$ and
a normal subgroup $U$ of $G$ such that $C=R\times U.$
We have $R=R_G(V),$ $C=C_G(V)$ where $V \in \mathcal V$ is an irreducible $G$-module  and $U\cong_G V^\delta$ for $\delta=\delta_G(V).$ Moreover
$G/R\cong U \rtimes H$
where $H$ acts in the same say on each of the $\delta$ factors of $U$ and this action
is faithful and irreducible. Let $r=\dim_{\End_G(V)}V$ and $\alpha=\lceil \delta/r \rceil.$
 Set $\beta= \eta(G/U)$ and note that
$$ \eta= \beta +\alpha.$$
 By induction,
$d_I(G/U)\leq \eta(G/U)= \beta.$
 In particular $G$ contains $\beta$ elements
$x_1,\dots,x_\beta$ that invariably generate $G$ modulo $U.$
 Now set $x_{\beta+1}= \ldots= x_{\eta}=1$: of course the $\eta= \beta +\alpha$ elements
$x_1,\dots,x_\beta,x_{\beta+1},\dots,  x_{\eta} $ are still invariable generators of $G$.
 Clearly for every $i=\beta+1, \dots, \eta $, we have that
$ \dim_{\End_G(V)}C_V(x_i)= \dim_{\End_G(V)}V=r$, hence
$$\sum_{1\leq i\leq \eta}\dim_{\End_G(V)}C_V(x_i)\geq \sum_{\beta+1\leq i\leq \eta}\dim_{\End_G(V)}C_V(x_i)\geq \alpha\cdot r\geq \delta$$
and therefore, by Proposition \ref{matrici}, there exist $u_1,\dots,u_{\eta} \in U$ such that $x_1u_1,\dots,x_\eta u_\eta$ invariably generate
$G$ modulo $R$. Let $g_1,\dots,g_{\eta}\in G$ and consider the subgroup $K=\langle (x_1u_1)^{g_1},\dots,(x_\eta u_\eta)^{g_\eta}\rangle.$
By construction, $x_1,\dots,x_{\eta}$ invariably generate $G$ module $U$, hence $KU=\langle x_1^{g_1},\dots,x_\eta^{g_\eta}\rangle U=G.$
Since $x_1u_1,\dots,x_\eta u_\eta$ invariably generate $G$ modulo $R$, we have also $KR=G.$ We conclude from Lemma \ref{sotto} that $G=K$.
Hence $x_1u_1,\dots,x_\eta u_\eta$ invariably generate
$G.$
\end{proof}

\section{Supersoluble groups}\label{Supersoluble}

We start this section with an easy observation.
\begin{prop}\label{nilpotent}A nilpotent group $G$ is coprimely-invariably generated if and only if it is cyclic.
\end{prop}
\begin{proof}Clearly an abelian group is CIG if and only if it is cyclic.
If $G$ is a nilpotent  CIG group, then $G/\frat(G)$ is abelian
 and CIG, hence $G/\frat(G)$, and consequently $G$, is cyclic.
\end{proof}

What happens if we replace nilpotency with supersolubility? How \lq\lq thin\rq\rq  \ is
a supersoluble CIG group?
 A first answer is given by Theorem \ref{super}, where it is proved that
 $\delta_G(V)=1$ for every complemented chief factor $V$ of $G$.


\begin{proof}[Proof of Theorem \ref{super}]
 Let $G$ be a finite supersoluble CIG group and assume by contradiction
that a chief series of $G$ contains two different complemented chief factors isomorphic to an irreducible
$G$-module $V$. Let $H=G/C_G(V)$. The semidirect product $V^2\rtimes H$ is an epimorphic image of $G$,
 hence
it is  CIG. It follows from Theorem \ref{abcase} that $2\leq \dim_{\End_H(V)}V,$ however
since $G$ is supersoluble, $V$ is cyclic and $\dim_{\End_H(V)}V=1,$ a contradiction.

 Conversely, assume that no chief series of $G$ contains two complemented and $G$-isomorphic chief factors. Let $p$ be the largest prime divisors of $G$: the Sylow $p$-subgroup $P$ of $G$ is normal.
By induction, a $p$-complement $H$ of $P$ in $G$ has a CIG-set, say $\{h_1,\dots,h_d\}$. Moreover
there exists $x$ in $G$ such that $\langle x \rangle^H=P.$ Indeed $P/\frat(P)\cong V_1\times \dots \times V_u$
where $|V_i|=p$ for each $i\in \{1,\dots,u\}$ and, by our hypothesis, $V_i$ and $V_j$ are not $H$-isomorphic
if $i\neq j$: this implies that $P/\frat(P)$ is a cyclic $H$-module. It can be easily verified that
$\{x,h_1,\dots,h_d\}$ is a CIG-set of $G.$
\end{proof}

 As a corollary, we have a bound on the minimal number of generators of the Sylow $p$-subgroups.

\begin{cor}Let $G$ be a  supersoluble and coprimely-invariably generated finite  group. Then for
any prime $p$ dividing $|G|,$ a Sylow $p$-subgroup of $G$ can be generated by $p-1$ elements.
\end{cor}

\begin{proof}
If $p_1<p_2<\dots<p_t$ are the prime divisors of $|G|$, then
there exists a normal series $1=N_t<\dots < N_0=G$ with the property that
$N_{i-1}/N_i$ is a Sylow $p_i$-subgroup of $G/N_i.$ Let $P_i=N_{i-1}/N_i$
and let $V_i=P_i/\frat{P_i}.$ Since $G$ is supersoluble, $V_i=U_{i,1}\times \dots \times U_{i,t_i},$
with $|U_{i,j}|=p_j$ for each $j\in\{1,\dots,t_i\};$ moreover, by Proposition \ref{super},
$U_{i,j_1}$ and $U_{i,j_2}$ are not $G$-isomorphic if $j_1\neq j_2.$
In particular $G/C_G(V_i)$
is abelian, and consequently cyclic by Proposition \ref{nilpotent}.
However a cyclic group has at most $p_i-1$ 1-dimensional inequivalent representations
over the field with $p_i$-elements,   so $d(P_i)=d(V_i)=t_i\leq p_i-1.$
\end{proof}

In the remaining part of this section we will prove that
the previous result is best possible: for any prime $p$ it
can be constructed a  supersoluble coprimely-invariably generated finite  $G$ with the property that $d(P)=p-1$ for every Sylow $p$-subgroup $P$
 of $G.$
Notice that this group is actually minimal-exponent.

\

In order to construct these examples we need to recall
some properties of the Sylow $p$-subgroup $P_m$ of $\perm(p^m)$
and its normalizer $N_m$ in $\perm(p^m).$
We have that $P_m$ is isomorphic to the iterated
wreath product $C_p\wr \dots\wr  C_p$ of $m$ copies of the cyclic group of order $p.$
In particular $P_m\cong B_m\rtimes P_{m-1}$ with $B_m\cong (\mathbb Z/p\mathbb Z)^{p^{m-1}}.$
Notice that $$I_m=\left\{(x_1,\dots,x_{p^{m-1}}) \in B_m \mid \sum_i x_i=0\right\}$$ is the
unique maximal $P_{m-1}$-submodule of $B_m.$ Let $y_m=(1,0,\dots,0) \in B_m.$
Clearly $B_m=I_m \times \langle y_m \rangle$. Moreover
$F_m=I_mI_{m-1}\cdots I_1$ is the Frattini subgroup of $P_m$
and $\gamma_1,\dots,\gamma_m$ is a basis for $P_m/F_m\cong C_p^m$, where $\gamma_i=y_iF_m.$
Let $N_m$ be the normalizer of $P_m$ in $\perm(p^m).$ We have that $N_m$ is a split extension
of $P_m$ with a direct product $H_m$ of $m$ copies of the cyclic group of order $p-1$ (see for example \cite{cl}).
From the description of  the action of $H_m$ over $P_m$ given in \cite{bg} before the statement of Proposition 2.4
it follows that for every $(\beta_1,\dots,\beta_m)\in (\mathbb Z/p\mathbb Z)^*,$ there exists $h \in H$ such that
$\gamma_i^h=\beta_i\gamma_i$ for every $1\leq i \leq m.$ In particular $B_m^h=B_m,$ so $h$ induces an automorphism
$\overline h$ of $P_m/B_m\cong P_{m-1}.$

\begin{lemma}Let $m\leq p-1$ and let $h\in H_m$ such that
$\gamma_1^h=\beta_1\gamma_1,\dots,\gamma_m^h=\beta_m\gamma_m$
with $\beta_i\neq \beta_j$ whenever $i\neq j.$
 Then the group $G=P_m\rtimes \langle h \rangle$ is minimal-exponent.
\end{lemma}
\begin{proof}
Let $\alpha=|h|.$ Since $\exp(P_m)=p^m,$ we have that $\exp(G)=p^m\alpha.$ Let $H\leq G$ with $\exp(H) = \exp(G).$
In particular $H$ must contain an element of order $\alpha:$ the subgroups of $G$ with order $\alpha$ are conjugate, so it is not
restrictive to assume $h \in H.$ Let $Q=P_m\cap H$ and consider $B=B_m \cap Q.$ Since $Q$ contains an element
of order $p^m$, the factor group $Q/B$ must contains an element of order $p^{m-1}$. Let $\bar h$ be the automorphism
group of $P_m/B_m$ induces by $h$. We have that $\overline H = QB_m/B_m\rtimes \langle \overline h \rangle$ is a subgroup of $\overline G= P_m/B_m\rtimes\langle \overline h \rangle$ with $\exp(\overline H)=\exp(\overline G).$ By induction   $\overline H=\overline G$, hence $QB_m=P_m.$ Let now
$x$ be a $p^m$-cycle contained in $Q.$ We can write $x=by$ with $b=(u_1,\dots,u_{p^{m-1}})\in B_m$ and
$y$ a $p^{m-1}$-cycle in $P_{m-1}.$ Since $|x|=p^m$ we must have $(by)^{p^{m-1}}\neq (0,\dots,0)$
and this implies $\sum_i u_i \neq 0$ i.e. $b \notin I_m.$ This means that $xF_m=x_1\gamma_1+\dots+ x_m\gamma_m$
with $x_1\neq 0.$ Applying the subsequent Lemma \ref{elinear} to the $\langle h \rangle$-invariant subspace
$W=QF_m/F_m$ of $V=P_m/F_m$, we deduce that $\gamma_m\in QF_m/F_m$: hence $P_m=QB_m=QF_m=Q$ and $H=G.$
\end{proof}

\begin{cor}\label{12} Let $p$ be an odd prime. The $p$-group $P_{p-1}$ admits an automorphism $h$ of order $p-1$
with the property that $G=P_{p-1}\rtimes \langle h \rangle$
 is a minimal-exponent group.
This group $G$ is supersoluble with exponent $(p-1)p^{p-1}$ and order $(p-1)p^{\frac{p^{p-1}-1}{p-1}}$.
\end{cor}

\begin{lemma}\label{elinear}Let $V=F^n$ be an $n$-dimensional vector space over the field $F$ and let $\alpha=\diag(\beta_1,\dots,\beta_n).$
Suppose that $\beta_i \neq 0$ for each $i\in \{1,\dots,n\}$ and that
$\beta_i \neq \beta_j$ whenever $i\neq j.$ Suppose that an $\alpha$-invariant
subspace $W$ of $V$ contains an element $v=(x_1,\dots,x_n)$ with $x_1\neq 0.$
Then $(1,0,\dots,0)\in W.$
\end{lemma}
\begin{proof}Let $\Omega:=\{(y_1,\dots,y_m)\in W \mid y_1 \neq 0\}$ and assume that $\omega=(z_1,\dots,z_n)$ is
an element of $\Omega$ with the property that the cardinality of the set $J_\omega=\{i\neq 1\mid z_i \neq 0\}$ is as
smallest as possible. Assume by contradiction that $J_\omega\neq \emptyset$ and choose $i\in J_\omega.$ We have
$$\overline \omega=\beta_i\omega-\omega^\alpha=(\beta_iz_1,\dots,\beta_iz_n)-(\beta_1z_1,\dots,\beta_nz_n)\in W.$$
Since $\beta_i-\beta_1\neq 0,$ we have $\overline \omega\in \Omega.$ Moreover $J_{\overline{\omega}} \subseteq J_\omega$
but $i \in  J_\omega \setminus J_{\overline{\omega}}$, against the minimality of $|J_\omega|.$
\end{proof}




\section{Further examples and some questions}\label{quest}

 In this section we prove that all the finite simple groups are PCIG and we make some further considerations   on the differences among  the properties  PCIG and CIG and the property
 of being minimal-exponent.

In \cite{ig} the authors prove that every non-abelian finite simple group $G$ is invariably generated by 2 elements;
for many simple groups, the two generators exhibited in \cite{ig} have coprime orders. Examining the remaining cases, in \cite{CIG} it is proved
that a finite nonabelian simple groups $G$ contains four elements of pairwise coprime order
invariably generating $G$ (in fact three element suffices if $G\neq \po^+_8(2),  \po^+_8(3)$).

\begin{thm}\label{simple} All the finite simple groups satisfies PCIG.
\end{thm}
\begin{proof}Let $G$ be a finite simple groups. Clearly if $G$ is cyclic of order $p,$ then $G$ is PCIG, so we may assume that $G$ is non abelian.
First consider the case $G=\alt(n)$. If $n\geq 7,$ then there exists a prime $p$ with $n/2 < p < n-2$
(e.g. by Nagura's result \cite{na}). Write $n=p_1^{n_1}\dots p_r^{n_r}$
as a product of powers of different primes  and consider
the following elements: $y,$ $x_1,\dots,x_r$ where $y$ has order $p$,
$x_i$ is the product of $n/p_i^{n_i}$ disjoint cycles of length $p_i^{n_i}$ if $p_i$ is odd,
$x_i$ is the product of $n/2^{n_i-1}$ disjoint cycles of length $2_i^{n_i-1}$ if $p_i=2.$
Let now $a,b_1,\dots,b_r \in \alt(n)$ and consider $H=\langle y^a,x_1^{b_1},\dots,x_r^{b_r}\rangle$ and let $\Delta$ be an $H$-orbit: we have
that $|x_i|$ divides $|\Delta|$ for each $i\in \{1,\dots r\}$ and $|\Delta|\geq p>n/2,$ hence $|\Delta|=n$. This implies
that $G$ is transitive (and consequently primitive since it contains a $p$-cycle): but
then $H=\alt(n)$ by \cite[Theorem 3.3E]{dm}, hence $y,$ $x_1,\dots,x_r$ invariably generate $\alt(n).$
If $n\in \{5,6\}$ then $\alt(n)$ is invariably generated by any subset $\{x_1,x_2,x_3\}$ with $x_1$ or order 5, $x_2$ of order 3 and $x_3$ of order $2$ if $n=5$, of order $4$ if $n=6$. To deal with the other simple groups we use \cite[Corollary 5]{107}: all the pairs $(G,M)$
where $G$ is a finite non abelian simple group and $M$ is a proper subgroup of $G$ with $\pi(G)=\pi(M)$ are listed in \cite[Table 10.7]{107}.
Now, for a given simple group $G$, consider a set $\Omega_G=\{x_p \mid p\in \pi(G)\},$ where $x_p$ is a nontrivial $p$-element of $G$.
If  $G$ does not appear in \cite[Table 10.7]{107}, then $\Omega_G$ is a PCIG-set of $G$, independently of
our choice of the elements $x_p.$ Let $\mathcal S$ be the set of simple groups that are not alternating groups and contain a proper subgroup $M$ with $\pi(G)=\pi(M).$ By \cite[Table 10.7]{107} if $G\in \mathcal S \setminus \{\po^+_8(2),\psl_6(2),\psp_6(2),\M_{12}\},$ then, up to conjugacy in $\aut(G),$ there exists a unique proper subgroup $M$ of $G$ with
$\pi(G)=\pi(M)$; moreover by \cite[Theorem 1.3]{Gi}, except when $G=\po^+_8(3),$ there exist a prime $p$
and a $p$-element $g\in G$ such that $g\notin \cup_{\phi \in \aut(G)}M^\phi$ (indeed $g$ can be chosen of order $p$
is $G\neq \M_{11}$, of order 8 otherwise). If $x_p=g, $ then $\Omega_G$ is a PCIG-set for $G.$
If $G\in \{\psl_6(2), \psp_6(2)\}$, then $G$ contains an element $G$ of order 9 while no proper subgroup $M$ of $G$ with
$\pi(G)=\pi(M)$ contains element of order 9: if we choose $x_3$ of order 9, then $\Omega_G$ is a PCIG-set.
The group $G=\M_{12}$ contains, up to conjugacy in $\aut(G)$, two subgroups $M$ with $\pi(G)=\pi(M)$:
$M_1 \cong \psl(2,11)$ and  $M_2 \cong \M_{11};$ on the other hand $G$ contains $g_1$ or order 2
and $g_2$ of order 3 with the property that $g_i\notin \cup_{\phi \in \aut(G)}M_i^\phi$ for $i=1,2.$
If we choose $x_2=g_1$ and $x_3=g_2$, then $\Omega_G$ is a PCIG-set.
The cases $G =  \po^+_8(2),  \po^+_8(3)$ are discussed in \cite{CIG}: they are invariably generated
by four elements of orders, respectively, 2,5,7,9 and 5,7,9,13.
\end{proof}

As remarked in the introduction, if $G$ is a  minimal-exponent group, then $G$ is PCIG, but the converse does not hold.
Moreover, it is easy to see that epimomorphic images of  PCIG groups  (or CIG groups) are PCIG  (CIG, respectively).
 Actually,  the converse holds when we consider quotients over Frattini subgroups.

\begin{lemma}A finite group $G$ is PCIG (CIG) if and only if $G/\frat(G)$ is PCIG (CIG).
\end{lemma}
\begin{proof}Let $F=\frat(G)$. Assume that $\{g_1F,\dots,g_dF\}$ is a PCIG-set of $G/F$ and let $|g_iF|=p_i^{n_i}.$
For every $i=1, \dots , d$, we can write $g_i=x_iy_i$ with $|x_i|$ a $p_i$-power and $|y_i|$ coprime with $p_i$; in particular $y_i\in F.$
It can be easily proved that $x_1,\dots,x_d$ is a PCIG-set of $G.$ The other implication is trivial. Similar arguments hold for CIG groups.
\end{proof}

On the other hand,  the property of being minimal-exponent  is not closed under epimorphic images, even when the kernel is contained in the Frattini
subgroup:
indeed, there exists a minimal-exponent group $G$ such that  $G/\frat(G)$ is not minimal-exponent.
For example, let  $G$ be the supersoluble minimal-exponent group  constructed in  Corollary \ref{12} for $p=3$:
then $G$ has
 order $3^4\cdot 2$ and exponent $18$. However, if $P$ is the Sylow 3-subgroup
of $G$, then $\frat(G)=\frat(P)$ and $G/\frat(G)$ is a group of order $18$ and exponent $6$
containing an element of order $6$: in particular $G/\frat(G)$ is not minimal-exponent.
Even the converse implication is false: there exists a finite group $G$ which is not  minimal-exponent
 but $G/\frat(G)$ is minimal-exponent.
For example, consider the dihedral group $K$ of order 12. Since $K/Z(K)\cong \GL(2,2)$ we have an action of
 $K$ on $H=C_2\times C_2$ with kernel $Z(K).$ Let $G$
be the semidirect product $H\rtimes K.$ Since $\exp(G)=\exp(K)=12$, $G$ is not minimal-exponent.
  However $\frat(G)=Z(K)$, hence $G/\frat(G)\cong \perm(4)$ is a minimal-exponent group.

\



Notice that every epimorphic image of  a  minimal-exponent group   is  PCIG.
 We have seen that examples of PCIG groups that are not minimal-exponent can be
 easily constructed. However we don't know the answer to the following
intriguing question:

\begin{question} Is it true that any PCIG group
  appears as an epimorphic image of a minimal-exponent group?
\end{question}

In the following we want to discusses some examples and analyze some particular
instances of the previous question.

\

The first case that one would like to solve is that of finite simple groups.
By Theorem \ref{simple} every finite simple group satisfies PCIG,
 however a finite non abelian
simple groups is not necessarily minimal-exponent (for example $\exp(\alt(n-1))=\exp(\alt(n))$
except when $n=p^t$ with $p$ an odd prime or $n=2^t+2$ and $n-1$ is a prime-power). It should be interesting
to answer to the following questions.

\begin{question}Is it true that for any simple group $S$ there exists a  minimal-exponent finite group $G$
 admitting $S$ as a composition factor?
\end{question}

\begin{question}Is it true that for any simple group $S$ there exists a  minimal-exponent finite group $G$
 admitting $S$ as an epimorphic image?
\end{question}

Consider for example $G:=\alt(8);$ $G$ is not  minimal-exponent, since $\exp(\alt(8))=\exp(\alt(7))=420$,
however $G$ is a composition factor of $\sym(8),$ which is a minimal exponent group.
 On the other hand, we don't know examples
of minimal-exponent groups admitting $\alt(8)$ as epimorphic image.
Similarly, $\exp(M_{11})=\exp(M_{12})$ 
 and $\exp(M_{23})=\exp(M_{24})$
and we don't know examples
of minimal-exponent groups admitting $M_{12}$ or $M_{24}$ as composition factor.

\

More in general, given a simple group $S$, we may define
$\beta_G(S)$ as the number of normal subgroups $N$ of $G$ with $G/N\cong S.$
If $G$ is a  CIG group and $S$ is abelian, then $\beta_G(S)\leq 1$ as a consequence of Proposition \ref{nilpotent}.
However the situation is different for non-abelian simple groups. At the end of this section we will show
that for any positive integer
$t,$ it can be exhibited a non abelian simple group $S$ such that $S^t$ is CIG.
It is not clear whether the number $\beta_G(S)$ can be arbitrarily large when $G$ is minimal-exponent.
One could be tempted to conjecture that $\beta_G(S)\leq 1$ if $G$ is minimal-exponent and $S$ is simple,
but this is wrong. Here we construct an example with $\beta_G(S)=2.$

\

\noindent{\bf{Example}}.
Let $T=\ssl(2,41).$ We have that $Z=Z(T)=\langle z \rangle$
where $z$ is the unique element of $T$ of order 2, moreover $T$
contains elements of order 8 and has exponent $8\cdot 3 \cdot 5 \cdot 7 \cdot 41.$ The factor group
$S=T/Z=\psl(2,41)$ has order $8\cdot 3 \cdot 5 \cdot 7 \cdot 41.$
The group $T$ has a transitive permutation representation $\varphi$ of degree 42 with $\ker\varphi=Z$; if $p\in \{3,5,7,41\}$,
then the wreath product $C_p\wr_\varphi T$ contains an element of order $p^2.$
 Now consider
$$X=(V_3\times V_5\times V_7 \times V_{41})\rtimes (T_1\times T_2)$$
where
\begin{itemize}
\item $T_1\cong T_2\cong \ssl(2,41),$ $V_p\cong C_p^{42},$
\item$T_1$ centralises $V_5$ and $V_7$ and acts on $V_3$ and $V_{41}$ permuting the 42 entries,
\item $T_2$ centralises $V_3$ and $V_{41}$ and acts on $V_5$ and $V_{7}$ permuting the 42 entries.
\end{itemize}
Let $Y=\langle (z,z)\rangle \leq Z(T_1\times T_2)$ and consider
$G=X/Y.$ Notice that
$\exp(G)=8\cdot 3^2 \cdot 5^2 \cdot 7^2 \cdot 41^2.$
Set $M=((V_3\times V_5\times V_7 \times V_{41})\times Z(T_1\times T_2))/Y\leq G$
and notice that  $G/M=S_1\times S_2$ with $S_i\cong PSL(2,41)$.
Let now  $H$ be a subgroup of $G$ with $\exp(G)=\exp(H)$. For $i=1,2$, consider the projections
 $\pi_i: HM/M\to S_i$ of $HM/M$ onto the $i$-th factor of $G/M=S_1\times S_2.$
Since $H$ contains elements of order $3^2$ and
$41^2$, $\pi_1(HM/M)$ must contain elements of order 3 and 41. However
if $U$ is a maximal subgroup of $S$, then $|U|\in \{41\cdot 20, 60, 42, 40, 24\},$
and therefore we deduce that $\pi_1(HM/M)=S_1.$
Furthermore, since $H$ contains elements of order $5^2$ and
$7^2$, we have that $\pi_2(HM/M)$ must contain elements of order 5 and 7 and therefore
$\pi_2(HM/M)=S_2.$ Now suppose by contradiction that $HM/M \neq S_1\times S_2$.
This means that $HM/M=\{(s,s^\phi)\mid s \in S\}$ for some $\phi \in \aut(S).$
However $H$ must contains an element $h$ of order $8$: it is not restrictive to assume that $h=(t_1,t_2)Z\in (T_1\times T_2)Z$
with $t_2Z(T_1)=(t_1Z(T_2))^\phi$. In particular $|t_1|=|t_2|=8.$ But then $(t_1,t_2)^4=(z,z) \in Y$
and $|h|=4,$ a contradiction.
 Thus, we have proved that if $H$ is a subgroup of $G$ with the same exponent, then  $HM/M = S_1\times S_2.$

Now let  $G^*$ be a  subgroup  of $G$ minimal with respect to the property that
$\exp(G^*)=\exp(G)$:
 then   $G^*$ is a minimal-exponent group and, by the above discussion, $\beta_{G^*}(S)=2.$

\

We don't know other examples substantially different from the previous one;
in particular we don't know whether it is possible to exhibit a  minimal-exponent group $G$
  with $\beta_G(S)=3$ for some non abelian simple
group $S$. The situation is different if we consider finite PCIG groups.
To construct our next examples, we need to introduce some more notations.

\

    Let $S$ be a non-abelian finite simple group.
    Define the set
    $$ \mathcal{X}_d=\{(x_1,\dots,x_d) \in S^d
     \mid \{x_1, \ldots ,x_d\} \ \textrm{invariably generates } S\} $$
     and consider the equivalence relation on $\mathcal{X}_d$:
 $(x_1,\dots,x_{d}) \sim
   (y_1, \dots, y_{d})$
if and only if there exist $(s_1,\dots,s_d)  \in S^d$ and
$\alpha \in \aut(S) \ \mathrm{ s.t. } \
    (x_1^{s_1\alpha},\dots,x_d^{s_d\alpha})
    =(y_1,\dots,y_d).$

\

For reader's convenience, we sketch the proof of the following elementary criterion.

\begin{lemma}\label{columns}
 Let $G=S^t$ for a non-abelian finite simple group  $S$ and an integer $t$.
For $i\in \{1,\dots,d\}$, let $g_i=(g_{1,i},\dots,g_{t,i})\in S^t.$
 Then  $\{g_1, \ldots , g_d\}$ invariably generates $G$ if and only if
$(g_{i,1},\dots,g_{i,d}) \in \mathcal{X}_d$ for all $i$ and
$(g_{i,1},\dots,g_{i,d})
\not\sim
(g_{j,1},\dots,g_{j,d})$ whenever $i\neq j.$
\end{lemma}

\begin{proof}
Consider the matrix $R$ whose columns correspond to the elements $g_1,\dots,g_d:$
$$R = \begin{pmatrix}
g_{1,1} & g_{1,2} & \cdots & g_{1,d} \\
g_{2,1} & g_{2,2} & \cdots &g_{2,d} \\
\vdots & \vdots& \ddots & \vdots \\
 g_{t,1} & g_{t,2} & \cdots & g_{t,d}
 \end{pmatrix}.$$
A necessary condition to have that $\{ g_1, \ldots , g_d\}$ invariably generates $G$,
is that each row of $R$ belongs to $\mathcal{X}_d$.
This is not enough: we must be sure that no subgroup generated by conjugates
of the elements $g_1,\dots,g_d$ is contained in the
diagonal subgroup  $\{ (h_1, \cdots , h_t) \in S^t \mid h_i^\alpha=h_j \}$
for some choice of $i\neq j$ and $\alpha\in \aut S.$ This is equivalent
to say that no pair of different rows of the matrix $R$ are equivalent.
\end{proof}

As a first application of this Lemma, we show that $\alt(5)^t$ is  CIG (and PCIG) if and only if $t \le 2$.

\begin{prop}
The group $\alt(5)^3$  is not CIG, while the group  $\alt(5)^2$ is PCIG.
\end{prop}
\begin{proof}
Let $S=\alt(5)$. The set of prime divisors of $|S|$ is $\{2,3,5\}$ and the representatives of its conjugacy classes are
 $1$,   $\tau=(1,2)(3,4)$, $\rho=(1,2,3)$, $\sigma=(1,2,3,4,5)$ and $\sigma^\alpha$  where $\alpha=(1,2) \in \aut S$.
The elements $\tilde{g_1}=(\sigma,\sigma)$,
 $\tilde{g_2}=(\rho,\rho)$ and $\tilde{g_3}=( \tau,1)$ invariably generate $\alt(5)^2$, and hence
$\alt(5)^2$ is PCIG and CIG.
Now consider the case of $G=S^3$ and assume by contradiction that
$G$ admits a CIG-set $X$. Note that any invariable generating set  of $S$ has to contain an element of order 5
(otherwise suitable conjugates of the elements of this set are all contained in the same point stabilizer) and an element
of order 3 (otherwise suitable conjugates of the elements of this set are all contained in the normalizer of a Sylow 5-subgroup).
This implies that it is not restrictive to assume that the elements of $X$ are the following:
$g_1=(\sigma^{\delta_1},\sigma^{\delta_2},\sigma^{\delta_3})  $ with $\delta_i \in \{1, \alpha\}$, $g_2=(\rho,\rho,\rho)$
and   $g_3=(\tau^{\epsilon_1},\tau^{\epsilon_2},\tau^{\epsilon_3})$  with $\epsilon_i \in \{0,1\}$. Since $(\tau^{\epsilon_i})^\alpha= (\tau^{\epsilon_i})^{s_1}$  and $\rho^\alpha=\rho^{s_2}$  for some $s_1, s_2 \in S$,
$(\sigma,\rho,\tau^{\epsilon_i})$ is equivalent to
$(\sigma^\alpha,\rho,\tau^{\epsilon_i}).$
Therefore, we have only two
 choices, up to equivalence, for the rows of the matrix $R$, whose columns are $g_1,g_2 , g_3$, namely
 $(\sigma,\rho,\tau)$ and
$(\sigma,\rho,1)$, but then by Lemma  \ref{columns},   $\{g_1, g_2 , g_3\}$ cannot invariably generate $\alt(5)^3$, a contradiction.
\end{proof}

We conclude by showing that for every $t\ge 1$ there is  a PCIG group $G$  with $\beta_S(G)\geq t$ for some choice of $S.$

\begin{prop}\label{S^t}
For any positive integer $t$ there exists a simple group $S$ with
the property that $S^t$ is PCIG.
\end{prop}
\begin{proof}
If $p$ is a large enough prime then $\alt(p)$ contains at least $t$ 2-elements $\rho_1,\dots,\rho_t$ which are not
pairwise conjugate in $\sym(p)$. Let $\sigma$ be a cycle of length $p$
and let $\tau$ be a cycle of length $3.$
Consider the elements $g_1=(\sigma,\dots,\sigma),$ $g_2=(\tau,\dots,\tau),$
$g_3=(\rho_1,\dots,\rho_t)$:  
 then $g_1, g_2 , g_3$
  satisfy the criterium of Lemma \ref{columns}, and hence $\{g_1, g_2 , g_3\}$ is a PCIG-set for $\alt(p)^t.$
\end{proof}


\end{document}